\documentclass[reqno,12pt]{amsart}
\usepackage{amsthm,amsmath,amsfonts,amssymb,amsthm,graphicx,epsf,epstopdf,color,pifont,flafter,bm,amscd,tabu}
\usepackage{graphicx}
\usepackage[margin=3.5cm]{geometry}
\usepackage{subcaption}

\newtheorem {theorem} {Theorem}

\newtheorem {definition}{Definition}
\newtheorem {corollary}{Corollary}

\newtheorem {example} {Example}
\newtheorem {remark}{Remark}

 \usepackage{enumitem}

\begin{document}


\title[Foliations admitting an invariant algebraic set]{On the holomorphic foliations admitting a common invariant algebraic set}

\author{Guangfeng Dong}

\address{Department of Mathematics, Jinan University,
	Guangzhou 510632, Guangdong province, China}


\email{donggf@jnu.edu.cn}
\author{Chujun Lu}

\address{Department of Mathematics, Jinan University,
	Guangzhou 510632, Guangdong province, China}
\email{1761516631@qq.com}


\begin{abstract}
In this paper, we study the holomorphic foliations admitting a common invariant algebraic set $C$ defined by a polynomial $f$ in $ \mathbb{K}[x_1,x_2,...,x_n]$ 
over any characteristic $0$ subfield $\mathbb{K}\subseteq\mathbb{C}$. 
 For the $\mathbb{K}[x_1,x_2,...,x_n]$-module $V_f$ of vector fields generating foliations that admit $C$ as an invariant set, we provide several conditions under which the module $V_f$ can be freely generated by a minimal generating set.
 In particular,  
when $n=2$  and $f$ is  a weakly tame polynomial, we show that the $\mathbb{K}[x,y]$-module $V_f$  is freely generated by two polynomial vector fields, one of which is the Hamiltonian vector field induced by $f$,  if and only if,  $f$ belongs to the Jacobian ideal $\langle f_x, f_y\rangle$ in $\mathbb{K}[x,y]$.  
Our proof employs a purely elementary method.
\end{abstract}






\maketitle

\section{Introduction and the main results}

The study of algebraic solutions to polynomial differential equations in the complex plane traces its intellectual origins to the seminal contributions of  19th to early 20th century mathematicians, notably including  Pfaff, Poincar{\'e}, Darboux, and Painlev{\'e}. In modern times, this line of inquiry aligns with the study of invariant algebraic sets in holomorphic foliations within complex projective spaces.

In \cite{Darb}, G. Darboux established a criterion for the existence of first integrals in polynomial differential equations, grounding his theory in the requisite abundance of invariant algebraic curves. J.P. Jouanolou’s work in \cite{Joua} improved and generalized Darboux’s theory of integrability characterizing rational first integrals.

In the complex projective plane $\mathbb{CP}^2$, when the invariant curves exist, 
H. Poincar{\'e} posed the following problem:
is it possible to bound the degree of an invariant curve in terms of the degree of the polynomial foliation?
In general, the answer is negative. However, by imposing conditions on the singularities, the answer becomes positive (see, e.g.,  \cite{Carn,Cer-LN} for some degree bounds). The complete characterization of the degree of invariant curves   remains largely open. For a review of this question and recent developments, see \cite{Correa}.

Invariant algebraic curves are also closely related to another well-known open problem: the second part of Hilbert's 16th problem. An important version of this problem involves studying the number and distribution of algebraic limit cycles for polynomial vector fields in the real plane. An algebraic limit cycle corresponds to a closed branch (in the real plane) of an invariant algebraic curve of the system. Relevant literature can be referred to \cite{A-CFL,GasGia,LlRS,Llb-Zhang,Zhang} and references therein.

In this paper we consider a rather fundamental problem which concerns the characterization of the algebraic structure of the space of polynomial foliations admitting a fixed invariant algebraic curve.
Let $C\subset \mathbb{C}^2$ be an affine algebraic curve  defined by $f(x,y)=0$, where $f(x,y) $ is a polynomial in $ \mathbb{C}[x,y]$. Recall that $C$ is an invariant curve of the foliation generated  by a vector field  
\(\mathcal{X} :=P(x,y)\frac{\partial}{\partial x}+Q(x,y)\frac{\partial}{\partial y}\)  
 (or equivalently, by a polynomial 1-form \(\omega:= Q(x,y) \mathrm{d}x -P(x,y)\mathrm{d}y\)) if and only if there exists a polynomial $k(x,y)\in \mathbb{C}[x,y]$ satisfying the equality
\begin{equation}\label{eq1}
	P f_x + Q f_y = k f,
\end{equation}
where  $f_x=\frac{\partial f}{\partial x}$, $f_y=\frac{\partial f}{\partial y}$.
The polynomial $k(x,y)$ is called the cofactor associated with the invariant curve  $f(x,y)=0$  for the vector field $\mathcal{X}$.  
The set of all polynomial vector fields admitting  $C$ as an invariant curve
 constitutes a $\mathbb{C}[x,y]$-module, denoted by $V_f$. 
Within this module, there are three trivial elements 
\[\mathcal{X}_0 = -f_y\frac{\partial}{\partial x}+f_x\frac{\partial}{\partial y},\ \ \mathcal{X}_1=f \frac{\partial}{\partial x}, \ \ \mathcal{X}_2=f \frac{\partial}{\partial y},\] 
with corresponding cofactors 
\[k_0=0,\ \  k_1 =f_x,\ \ k_2=f_y,\] 
respectively. Clearly, $\mathcal{X}_0$ is a Hamiltonian vector field induced by $f$.
In the terminology of differential forms, the module $V_f$ is replaced by a  $\mathbb{C}[x,y]$-module $E_f=\{Q(x,y) \mathrm{d}x -P(x,y)\mathrm{d}y\}$, which is isomorphic to $V_f$.

Some characterizations of polynomial vector fields  possessing a given set of invariant algebraic curves  can be found in \cite{Llibre} and references therein. There are also results describing particular types of vector fields admitting a given invariant algebraic curve, such as Li\'enard systems, Kukles systems and others (see, e.g., \cite{Gine-Lli,Lli-Val}).

Within the aforementioned  framework, the field \(\mathbb{C}\) may be substituted with any characteristic $0$ subfield
 $\mathbb{K} \subseteq \mathbb{C}$. 
 A singularity $p=(x,y)\in \mathbb{K}^2$ on curve $C$ defined by $f\in \mathbb{K}[x,y]$ is a point satisfying $f_x(x,y)=f_y(x,y)=0$.  
We say a singularity of curve $C\subset \mathbb{K}^2$ is quasi-homogeneous, if it is quasi-homogeneous  on $ \mathbb{C}^2$ (see Definition \ref{def-quasi-s} below).

A natural question is to find a generating set of minimal cardinality for the
 \(\mathbb{K}[x,y]\)-module \(V_f\). This generating set generally depends 
 on the curve $C$ and the field $\mathbb{K}$.  
 Furthermore, it is interesting to study when $V_f$ is a free module and what its rank is.
 In \cite{Ca-Mo}, it is established that when  \( f\) is a weakly tame polynomial with at most quasi-homogeneous singularities, 
  the \(\mathbb{K}[x,y]\)-module \(V_f\) admits a generating set of cardinality at most $4$. Three elements of this generating set are the trivial vector fields  \(\mathcal{X}_0\), \(\ \mathcal{X}_1\) and \(\mathcal{X}_2\), while the unique non-trivial generator is constructed algorithmically.
Furthermore, the authors demonstrate that the 
\(\overline{\mathbb{K}}[x,y]\)-module 
\( V_f  \otimes_{\mathbb{K}}  \overline{\mathbb{K}} \) is free of  rank $2$ (which is minimal),
where \( \overline{\mathbb{K}}\) is the algebraic closure of \(\mathbb{K}\). 
Notably, explicit expressions for the minimal generating set remain unspecified in the general case.
In this paper, we will show that the above results can be strengthened in some aspects by the following theorem.

\begin{theorem}\label{th-main}
For any algebraic curve $C$ defined by a weakly tame polynomial $f(x,y) \in \mathbb{K}[x,y]$,  the $\mathbb{K}[x,y]$-module $V_f$ 	is freely generated by the Hamiltonian vector field $\mathcal{X}_0$ and an additional generator $\mathcal{X}_{*}$ if and only if  $f$ belongs to the Jacobian ideal $J(f):=\langle f_x, f_y\rangle$ in $\mathbb{K}[x,y]$.
\end{theorem}

\begin{remark}
	If a vector field  $\mathcal{X} $ can be generated by the other  ones $\mathcal{Y}_1$ and $\mathcal{Y}_2$, i.e., $ \mathcal{X}= R_1 \mathcal{Y}_1 + R_2 \mathcal{Y}_2,$ where $ R_1, R_2 \in \mathbb{K}[x,y]$, 
then the corresponding cofactors $k_{\eta}$, $k_{\eta_1}$ and $k_{\eta_2}$ also satisfy $k_{\eta}=R_1 k_{\eta_1}+R_2 k_{\eta_2}$.
All cofactors associated to $C$ for all vector fields in $V_f$ constitute an ideal $\mathcal{I}_{co}$ in $\mathbb{K}[x,y]$. Obviously, the Jacobian  ideal $J(f)$ is a subideal of $\mathcal{I}_{co}$.
\end{remark}



For an algebraically closed field $\mathbb{K}$, a singular curve $C$ defined by $f\in \mathbb{K}[x,y]$ has only quasi-homogeneous singularities if and only if \( f\in J(f)\) (see, e.g., \cite{CLOS}). Thus, in this case, we have the following  corollary by Theorem \ref{th-main}.

\begin{corollary}\label{cor-1}
	Assume that  $C$ is a singular curve defined by a weakly tame polynomial $f\in \mathbb{K}[x,y]$. Then $C$ has only quasi-homogeneous singularities (in $ \overline{\mathbb{K}}^2$) if and only if the  \({\overline{\mathbb{K}}}[x,y]$-module $V_f\)  is freely generated by two generators, one of which is \(\mathcal{X}_0\).
\end{corollary} 

This corollary can be regarded as a global version of   K. Saito's classical theorem (Theorem \ref{th-Saito-quasi} below) for arbitrary algebraically closed fields.
From the proof of Theorem \ref{th-main} below, the vector field \( \mathcal{X}_{*}\) can be determined by the decomposition of $f$ in the ideal \(J(f)\). Namely, if \(f=P f_x +Qf_y\) in \(J(f)\), then \( \mathcal{X}_{*}=P\frac{\partial}{\partial x}+Q\frac{\partial}{\partial y} \).

Note that if $ f$ is a quasi-homogeneous polynomial, then $C$ has only quasi-homogeneous singularities. However, the converse does not hold in general.
By Theorem \ref{th-main}, we can also obtain the following result for any field $\mathbb{K}$ (not merely for algebraically closed field). For a weakly tame polynomial $f$, it is quasi-homogeneous if and only if the  \({\mathbb{K}}[x,y]$-module $V_f\)  is freely generated by two generators \(\mathcal{X}_0\) and $\mathcal{X}_{*}= {l}x\frac{\partial}{\partial x} +{m}y\frac{\partial}{\partial y}$, where $(l,m)$ are the weights of $f$ (see Definition \ref{de-quasi} below). This generalizes Example 4 in \cite{Ca-Mo} for the tame polynomial $f$.

It is worth noting that Theorem \ref{th-main} is not valid for all smooth curves, since \( f\in J(f)\) may not hold for some of them. However, when \(f(x,y)=x+g(y)\) (resp. \(f(x,y)=y+g(x)\)), we have  \( f\in J(f)\). In this case, \(\mathcal{X}_{*}=\mathcal{X}_{1}\) (resp.  \(\mathcal{X}_{*}=\mathcal{X}_{2}\)), where \(g(y)\) is a polynomial in \(y\) (resp. \(g(x)\) is a polynomial in \(x\) ).
 
 For a general polynomial $f$ (not necessarily weakly tame), results on the conditions  for the $\mathbb{K}[x,y]$-module $V_f$ to have rank $2$ are scarce, to the best of our knowledge. 
 Let $D$ denote the greatest common divisor of  $f_x$ and $f_y$. Then $\mathcal{X}_0/D$ belongs to $ V_f$. We thus obtain  the following results. 
 
 \begin{theorem}	\label{th-main2}
For any algebraic curve $C$ defined by a  polynomial $f(x,y) \in \mathbb{K}[x,y]$,  if the $\mathbb{K}[x,y]$-module $V_f$ 	is freely generated by the  polynomial vector fields $\mathcal{X}_0/D$ and an additional generator $\mathcal{X}_{*}$, then  $f$ belongs to the ideal $\langle f_x/D, f_y/D\rangle$ in $\mathbb{K}[x,y]$.
\end{theorem}

The converse of the above theorem does not generally hold. However, we have the following result.

 \begin{theorem}	\label{th-main3}
For any algebraic curve $C$ defined by a  polynomial $f(x,y) \in \mathbb{K}[x,y]$,  
if  $f$ belongs to the Jacobian ideal $J(f)=\langle f_x, f_y\rangle$ in $\mathbb{K}[x,y]$, then 
the $\mathbb{K}[x,y]$-module $V_f$ 	is freely generated by the polynomial vector field $\mathcal{X}_0/D$ and an additional generator $\mathcal{X}_{*}$.
\end{theorem}
Similar to Corollary \ref{cor-1}, we can also obtain the following result for a singular curve with only quasi-homogeneous singularities.
\begin{corollary}\label{cor-2}
Assume that  $C$ is a singular curve defined by  $f\in \mathbb{K}[x,y]$. If $C$ has only quasi-homogeneous singularities on $\overline{\mathbb{K}}^2$, then the  \({\overline{\mathbb{K}}}[x,y]$-module $V_f\)  is freely generated by two generators, one of which is \(\mathcal{X}_0/D\).
	
\end{corollary}

The above problem is also meaningful in $\mathbb{K}^n$ of higher dimension $n\geq 3$.  In \cite{DMR} the authors give the normal form of all polynomial differential systems
in $\mathbb{R}^3$ having a weighted homogeneous surface $f= 0$ as an invariant algebraic surface
and characterize among these systems those having a Darboux invariant constructed
uniquely using this invariant surface.
However, to the best of our knowledge, the related results over arbitrary field $\mathbb{K}$ remain scarce. In Section \ref{sec-nd} we apply our method to this case and present some conditions that guarantee the module of vector fields admitting a given hypersurface is freely generated by a  generating set of minimal cardinality $n$ over any field $\mathbb{K}$.

The paper is organized as follows. Section \ref{sec-pre} introduces fundamental definitions and theorems. In Section \ref{sec-2d} we prove the results on $\mathbb{K}^2$ and give the details of our elementary method. Section \ref{sec-nd} addresses the case of $\mathbb{K}^n$ of dimension $n\geq 3$.

 \section{Preliminaries}\label{sec-pre}
Regarding the definitions of weakly tame polynomials and quasi-homogeneous singularities, we maintain terminological consistency with \cite{Ca-Mo}. 

\begin{definition}\label{de-quasi}
	We say that $f\in \mathbb{K}[x,y]$ is a quasi-homogeneous polynomial of degree $d$ with weights $ (l,m)$, if there exist  three non-negative integers $l$, $m$ and $d$ such that $f(t^l x, t^m y)=t^d f(x,y)$ for any $(x,y)\in \mathbb{K}^2$. 
\end{definition}

A polynomial $f\in \mathbb{K}[x,y]$ has a quasi-homogeneous decomposition of
$f=f_d +\cdots +f_2 +f_1 +f_0$ into degree $i$ quasi-homogeneous pieces $f_i$ with given weights $(l,m)$. 

\begin{definition}\label{de-tame}
A polynomial $f\in \mathbb{K}[x,y]$ is called tame if the highest degree $d$ quasi-homogeneous pieces $g:=f_d$ has finite dimensional Milnor vector space $ \mathbb{K}[x,y]/J(g)$.
A polynomial $f\in \mathbb{K}[x,y]$ is called weakly tame if the Milnor vector space $\mathbb{K}[x,y]/J(f)$ is finite dimensional.	
\end{definition}

The following theorems come from K. Saito's classical results over the field $\mathbb{C}$ (see \cite{Saito1,Saito2}), and they describes the local properties of quasi-homogeneous singularities and $\mathcal{O}_{\mathbb{C}^2,p}$-module, where  $\mathcal{O}_{\mathbb{C}^2,p}$ denotes the ring of germs of holomorphic functions at $p$.  

\begin{definition}\label{def-quasi-s}
Assume the origin $O\in \mathbb{C}^2$ is a singularity of  $f(x,y)$, where $f\in \mathcal{O}_{\mathbb{C}^2,O}$. The singularity $O$ is called quasi-homogeneous  if there is a holomorphic change of coordinates $(u,v)\in \mathbb{C}^2$ in a neighborhood of $O$ such that $(u(0),v(0))=(0,0)$ and  $f(x(u,v),y(u,v))=h(u,v)g(u,v)$, where $h$ is holomorphic satisfying $h(0,0)\neq 0$	and $g$ is a quasi-homogeneous polynomial in $u$ and $v$.  
\end{definition}

\begin{theorem}[\cite{Saito1}]\label{th-Saito-quasi}
 A germ of curve singularity $f=0$ is quasi-homogeneous if and only if $f$ belongs to the Jacobian ideal $J(f)$ in  $\mathcal{O}_{\mathbb{C}^2,p}$.	
\end{theorem}

\begin{theorem}[\cite{Saito2}]\label{th-Saito2}
Assume $f\in \mathcal{O}_{\mathbb{C}^2,p}$. The $\mathcal{O}_{\mathbb{C}^2,p}$-module of holomorphic 1-forms tangent to $f=0$ is 
 freely generated if and only if it has two elements $\eta_0$ and $\eta_{\infty}$ such that $\eta_0 \wedge \eta_{\infty}=f \mathrm{d}x \wedge \mathrm{d}y$.	
\end{theorem}

 \section{Proof of Theorem \ref{th-main} to Theorem \ref{th-main3}}\label{sec-2d}
 
The sufficiency part of Theorem \ref{th-main} can be obtained from Remark 4 in \cite{Ca-Mo} through Theorem \ref{th-Saito-quasi}, Theorem \ref{th-Saito2} and sheaf theory. Here we provide an elementary proof using only linear algebra. 

\begin{proof}[Proof of Theorem \ref{th-main}]
	To obtain the necessity, we only need to show that if  $V_f$ can be generated by $\mathcal{X}_0 $ and $\mathcal{X}_{*}$, then 
	the cofactor $k_{*}$ associated with $\mathcal{X}_{*}$  must be a nonzero constant $ c \in \mathbb{K}$.
	Let $\mathcal{X}_{*}=P_{*}\frac{\partial}{\partial x} +Q_{*}\frac{\partial}{\partial y}$. Then there exist  polynomials $R_{ij}\in\mathbb{K}[x,y]$, $i,j=1,2,$ such that 
	\begin{equation}
		\begin{array}{rcl}
		\mathcal{X}_1&=& R_{11}\mathcal{X}_0+R_{12}\mathcal{X}_{*},\\
		\mathcal{X}_2&=& R_{21}\mathcal{X}_0+R_{22}\mathcal{X}_{*}.	
		\end{array}
	\end{equation}
	By comparing the cofactors, we get
	\begin{equation}\label{eq3}
		\begin{array}{lcl}
		f_x= R_{12}k_{*},\ \ f_y=R_{22}k_{*}.& 
		\end{array}
	\end{equation}
That is, 	$k_{*} $ is a common factor of $f_x$ and $f_y$. Since for a weakly tame polynomial $f$ the partial derivatives  $f_x$ and $f_y$ are coprime, it follows that  $k_{*} $ is a nonzero constant. Therefore, $f=(P_{*} f_x+Q_{*} f_y)/k_{*} \in J(f)$.

For sufficiency, since $f\in J(f)$, we have  $f=P_{*} f_x+Q_{*} f_y$, where \( P_{*}, Q_{*}\in \mathbb{K}[x,y]\). Then $\mathcal{X}_{*}:=P_{*}\frac{\partial}{\partial x} +Q_{*}\frac{\partial}{\partial y}$ belongs to $V_f$ and has cofactor $k_{*}=1$.
It suffices to prove that for any $\mathcal{X}:=P\frac{\partial}{\partial x} +Q\frac{\partial}{\partial y}\in  V_f$ having cofactor $k$,   there exists a unique pair of polynomials $R_0$ and $R_{*}$ in $\mathbb{K}[x,y]$ such that $\mathcal{X}=R_0 \mathcal{X}_0 +R_{*} \mathcal{X}_{*}$.

Denote by $\Lambda$ the set of points $p=(x,y)\in \mathbb{C}^2$ such that $f_x(x,y)=f_y(x,y)=0$. Since $f$ is a weakly tame polynomial, $\Lambda$ is a finite set. 
Rewrite the equality \eqref{eq1} as follows: 
\[ f_x P + f_y Q - f k = 0. \]
For any point $p_0 = (x_0, y_0) \in \mathbb{K}^2$, consider the 3-tuples
\[\alpha(p_0) := \big( f_x(x_0, y_0), \, f_y(x_0, y_0), \, f(x_0, y_0) \big)\]
and
\[\beta(p_0) := \big( P(x_0, y_0), \, Q(x_0, y_0), \, -k(x_0, y_0) \big)\]
as vectors in the vector space $\mathbb{K}^3$.
 Then $\beta(p_0)$ is a solution to the system of linear equations
 \begin{equation}\label{eq-main}
  f_x(x_0,y_0)X+ f_y(x_0,y_0)Y+ f(x_0,y_0)Z=0.	
 \end{equation}
Moreover, the solution space $L(p_0)$ of this system is a vector space $\mathbb{K}^2$ at any point $p_0\in \mathbb{K}^2\setminus \Lambda$.

Note that two vectors \[\beta_0(p_0)=(-f_y(x_0,y_0), f_x(x_0,y_0),0)\]
and 
 \[\beta_{*}(p_0)=(P_{*}(x_0,y_0), Q_{*}(x_0,y_0), -1)\]
  associated with  $\mathcal{X}_0$ and $\mathcal{X}_{*} $
form a basis of $L(p_0)$ at each  point $p_0\in \mathbb{K}^2\setminus \Lambda$. 
Therefore, the vector $ \beta(p_0)$ associated with a vector field $\mathcal{X} \in V_f$ can be  uniquely decomposed into a linear combination of $\beta_0(p_0) $ and $\beta_{*}(p_0)$. Namely, there exists a unique pair of numbers $ R_{0}(p_0)$ and $R_{*}(p_0)$ in $\mathbb{K}$ such that 
\[ \beta(p_0)=R_{0}(p_0)\beta_0(p_0) +R_{*}(p_0)\beta_{*}(p_0).\]
This defines two functions $R_0$ and $ R_{*}$ from $\mathbb{K}^2\setminus \Lambda $ to $\mathbb{K}$ such that $ \mathcal{X}=R_0 \mathcal{X}_0 +R_{*} \mathcal{X}_{*}$.  Thus, on $\mathbb{K}^2\setminus \Lambda $, we have
\begin{equation}\label{eq-R}
	\begin{array}{rcl}
		R_{*}  &=& k,\\
		-f_y R_0 +P_{*} R_{*}&=& P,\\
	f_x R_0 +Q_{*} R_{*}&=&Q.		 
	\end{array}
\end{equation} 
Clearly, $ R_{*}$  coincides with the polynomial  $k$ at every point of $\mathbb{K}^2\setminus \Lambda $. If the equality $ R_{*}=k$ can be extended to $\Lambda$, then by the equalities \eqref{eq-R}, we can obtain that 
\begin{equation}
	\begin{array}{rcl}
		R_{0}  &=& \dfrac{kP_{*}-P}{f_y}=\dfrac{Q-kQ_{*}}{f_x},	 
	\end{array}
\end{equation} 
which implies $ f_x (kP_{*}-P)=f_y(Q-kQ_{*})$. Using the weak tameness of $f$, we conclude that $f_x$ divides $ Q-kQ_{*} $ and  $f_y$ divides $ kP_{*}-P $, i.e., $R_0$ is a polynomial. 

To complete the proof, all that remains is to demonstrate that the equality $ R_{*}=k$ can be extended to $\Lambda$. 
For this, we deal with the above functions in $\mathbb{C}$. By similar arguments, the equality $R_{*}=k$ can be defined on $\mathbb{C}^2\setminus\Lambda$. According to Hartogs' Extension Theorem, $R_{*}$ can be holomorphically extended to $\mathbb{C}^2$, i.e., \(R_{*}\in \mathbb{C}[x,y]\). In addition, $R_{*}$ takes values in $\mathbb{K}$ on  $\mathbb{K}^2\setminus\Lambda$, thus it is a polynomial in \(\mathbb{K}[x,y]\), i.e.,  $R_{*}=k$ holds on $\mathbb{K}^2$.
The proof is finished.
\end{proof}

\begin{proof}[Proof of Theorem \ref{th-main2}]
Assume that $V_f$ is generated by $\mathcal{X}_0/D $ and $\mathcal{X}_{*}=P_{*}\frac{\partial}{\partial x} +Q_{*}\frac{\partial}{\partial y}$ having 
	the cofactor $k_{*}$.  Following the method used in the proof of the necessity part of Theorem \ref{th-main}, we conclude that $k_{*} $ is a common factor of $f_x$ and $f_y$, i.e., $k_{*}|D$. Therefore, we have $$f=\frac{P_{*} f_x+Q_{*} f_y}{k_{*}}=P_{*}\frac{D}{k_{*}}\frac{f_x}{D}+Q_{*}\frac{D}{k_{*}}\frac{f_y}{D}\in\left\langle \frac{f_x}{D},\frac{f_y}{D}\right\rangle.$$
\end{proof}

\begin{proof}[Proof of Theorem \ref{th-main3}]
The proof is similar to that for the sufficiency part of Theorem \ref{th-main}. Note that in this case, the equation \eqref{eq-main} becomes 
\begin{equation*}
  \frac{f_x}{D_f}(x_0,y_0)X+ \frac{f_y}{D_f}(x_0,y_0)Y+ \frac{f}{D_f}(x_0,y_0)Z=0,	\end{equation*}
  where $D_f$ is the greatest common divisor of $f_x$, $f_y$ and $f$.
By replacing 
 $\mathcal{X}_0$ with  $\mathcal{X}_0/D$, and the set $\Lambda$  with the set  $$\left\{(x,y)\in \mathbb{C}^2: \ \frac{f_y}{D}= \frac{f_x}{D}=0\right\}\cup \left\{(x,y)\in \mathbb{C}^2: \ \frac{f_y}{D_f}= \frac{f_x}{D_f}=\frac{f}{D_f}=0\right\},$$ 
which is also a finite set, one can obtain the conclusion immediately.
\end{proof}

\section{Modules of vector fields on $\mathbb{K}^{n\geq 3}$}\label{sec-nd}

In this section, we will employ the method in the proof of Theorem \ref{th-main} to study one-dimensional foliations admitting invariant hypersurfaces in higher-dimensional spaces.  

Let $f(x_1,...,x_n)$ be a polynomial in $\mathbb{K}[x_1,...,x_n]$,  and $C$ be the algebraic set given by $f=0$. Denote by $V_{f}$  the $\mathbb{K}[x_1,...,x_n]$-module  consisting of the vector fields generating the one-dimensional foliations that admit $C$ as an invariant set. Namely, for each element $\mathcal{X}=\sum_{j=1}^{n}P_j \frac{\partial}{\partial x_j}\in V_f$, there exists a polynomial $k\in \mathbb{K}[x_1,...,x_n]$ such that $$ \sum_{j=1}^{n}P_j f_{x_j} =kf.$$

Note that the subset of $V_f$ consisting of vector fields with cofactor $k\equiv 0$ forms a submodule of $V_f$, denoted by $V_f^0$. 
Clearly,  the minimal generating set of $V_f^0$ has cardinality at least $n-1$, while that of $V_f$ has cardinality at least $n$. Assume  $V_f^0$ contains  $n-1$ vector fields $\mathcal{X}_j=\sum_{m=1}^{n}P_{jm} \frac{\partial}{\partial x_m}$, $j=1,2,...,n-1$.
Let $\mathrm{M}$ be the coefficient matrix of $\{\mathcal{X}_j\}$, i.e., 
$$\mathrm{M}=\left( \begin{matrix}
	P_{11}   &\cdots &P_{(n-1)1} \\
	P_{12}  &\cdots & P_{(n-1)2}\\
	\vdots  & \ddots &\vdots \\
	P_{1n} & \cdots &P_{(n-1)n}
\end{matrix} \right),$$ 
and $M_j$ be the $(n-1)$-minor by deleting the $j$-th row of $\mathrm{M}$, respectively.
Then we have the following result.

\begin{theorem}\label{th-n1}
	If $f$ satisfies the following three conditions:
	\begin{enumerate}
  \item the rank of $\mathrm{M}$ is $n-1$ at each point of $\mathbb{K}^n\setminus \Lambda_0$, where $\Lambda_0$ is an algebraic set of dimension  $\leq n-2$, 
  \item the greatest common divisor of $\{M_1,...,M_n\}$ is $1$,
   
  \item $f\in \langle f_{x_1},...,f_{x_n}\rangle$,
\end{enumerate}
then the $\mathbb{K}[x_1,...,x_n]$-module $V_f$  can be freely generated by $ \mathcal{X}_1,..., \mathcal{X}_{n-1}$ and 
$\mathcal{X}_{*}:=\sum_{m=1}^{n}Q_m \frac{\partial}{\partial x_{m}}\in V_f$, where $\{Q_m\}$ are the coefficients of the decomposition of $f$ in $\langle f_{x_1},...,f_{x_n}\rangle$, i.e., $f=\sum_{m=1}^{n}Q_m f_{x_m}$.
\end{theorem}

\begin{proof}
Note that the cofactor $k_* $  of $\mathcal{X}_{*}$ is $1$. 
	The proof is similar to the sufficiency part of Theorem \ref{th-main}. 
	Let $\Lambda$ be the union of $\Lambda_0$ and the following set 
	\[\left\{(x_1,...,x_n)\in \mathbb{C}^n: \frac{f_{x_1}}{D}=\cdots =\frac{f_{x_n}}{D}=\frac{f}{D}=0\right\},\]
	where $D$ is the greatest common divisor of $f_{x_1},...,f_{x_n}$ and $f$.
	Clearly,  the dimension  of $\Lambda$ is also not greater than $n-2$.
	At each point of $\mathbb{K}^n\setminus \Lambda$, the vectors $\{(P_{j1},...,P_{jn},0), \ j=1,2,...,n-1\} $ and $(Q_1,...,Q_n,-1)$ form a basis of the solution space of the linear equation 
	\[ \sum_{m=1}^{n} \frac{f_{x_m}}{D} X_m +\frac{f}{D}X_{n+1}=0.\] 
	 For any vector field $\mathcal{X}=\sum_{j=1}^{n}\tilde{P}_j \frac{\partial}{\partial x_j}\in V_f$ with cofactor $\tilde{k}$, we can  uniquely define $n$ functions $ R_{1}, ..., R_{n-1}$ and $R_{*}$,  such that  $\mathcal{X}=\sum_{j=1}^{n-1}R_{j} \mathcal{X}_{j}+ R_{*} \mathcal{X}_{*}$ on $\mathbb{K}^n\setminus \Lambda$.  First,   Hartogs' Extension Theorem on $\mathbb{C}^n$ implies $R_{*}\in \mathbb{K}[x_1,...,x_n]$. Second,  Cramer's rule establishes that the functions $ R_{1}, ..., R_{n-1}$ are rational. Taking $ R_{1}$ as an example, when $M_j\not \equiv 0$ we have $ R_{1}={\tilde{M}_j}/{M_j}$, where $ \tilde{M}_j$ is a determinant whose elements are all polynomials. 
	 Finally, we show that the functions $ R_{1}, ..., R_{n-1}$ are polynomials. 
	 Let $ {M_{j_1},\cdots, M_{j_m}}$ be all the non-vanishing minors of $\mathrm{M}$. 
 The set  $\{p\in \mathbb{K}^n\, |\, {M_{j_1}\cdots M_{j_m}}\neq  0\}$ 
	 is a non-empty open set. Consequently, at each point of this set,  we have
	    \[R_{1}=\frac{\tilde{M}_{j_1}}{M_{j_1}}=\cdots =\frac{\tilde{M}_{j_m}}{M_{{j_m}}}. \]
	  Condition (2)  then implies $R_1$ must be a polynomial. This holds similarly for $ R_2,...,R_{n-1}$.  
\end{proof}
From the proof of the above theorem, one can also obtain that 
 $V_f^0$ is freely generated by $ \mathcal{X}_1,..., \mathcal{X}_{n-1}$ and $V_f=V_f^0 \oplus V_{f}^{*}$, where $V_{f}^{*}=\langle \mathcal{X}_*\rangle$. Theorem \ref{th-n1} may be regarded as the higher-dimensional counterpart of Theorem \ref{th-main3}, while the generalization of Theorem \ref{th-main2} in higher dimensions reads as follows:

\begin{theorem}
	If $V_f=V_f^0 \oplus V_{f}^{*}$, where $V_{f}^{*}$ is a free submodule of rank $1$, then  $f\in \langle f_{x_1}/D,...,f_{x_n}/D\rangle$, where $D$ is the greatest common divisor of $\{f_{x_1},...,f_{x_n}\}$.
	
\end{theorem}

\begin{proof}
	The proof is similar to the necessity part of Theorem \ref{th-main}, by observing that $ f\frac{\partial}{\partial x_j}$ belongs to $ V_f$ and has cofactor  $ f_{x_j}$.
\end{proof}




\begin{example}
For any characteristic $0$ subfield $\mathbb{K}\subseteq \mathbb{C}$ and  any $1<n\in \mathbb{Z}$,	let $f(x_1,\cdots,x_n)=x_1 +g(x_2,\cdots,x_n)$, where $g$ is an arbitrary polynomial in $\mathbb{K}[x_2,\cdots,x_n]$. Then $f$ satisfies the conditions of Theorem \ref{th-n1}, and $V_f$ is freely generated by $\{ \mathcal{X}_j=-\frac{\partial g}{\partial x_j}\frac{\partial}{\partial x_1} +\frac{\partial}{\partial x_j }, \ j=2, \cdots,n\}$ and  $ \mathcal{X}_*=f\frac{\partial}{\partial x_1} $.
\end{example}




	\section*{Acknowledgement}
We are grateful to Dr. Yuzhou Tian and Dr. Shaoqing Wang for their valuable comments.




\end{document}